\documentclass[12pt]{article}
\usepackage{amsthm, amsmath, amssymb,graphicx}

\usepackage{geometry}
\newcommand{\Sym}{\mathbb{S}}
\newcommand{\transp}{^\top}

\newcommand{\K}{{\cal K}}
\newcommand{\A}{{\cal A}}
\newcommand{\Aff}{{\rm Aff\,}}
\newcommand{\Lin}{{\rm Lin\,}}

\newtheorem {teo}{Theorem}
\newtheorem {lema}[teo]{Lemma}
\newtheorem {prop}[teo]{Proposition}
\newtheorem {cor}[teo]{Corollary}

\newtheorem {defi}[teo]{Definition}

\newtheorem {ej}[teo]{Example}

\def\xh{\hat{x}}
\def\yh{\hat{y}}

\def\sb{\bar{s}}

\def\yb{\bar{y}}

\def\calk{{\cal K}}

\def\A{\mathcal{A}}

\def\K{\mathcal{K}}

\def\SS{\Sym}

\def\eps{\varepsilon}

\def\intt{\mathop{\rm int}}

\def\val{\mathop{\rm val}}

\def\Min{\mathop{\rm min}}
\def\Max{\mathop{\rm max}}

\def\val{\mathop{\rm val}}

\def\1B{{\bf  1}}

\def\R{\mathbb{R}}

\newcommand{\rr}{\mathbb{R}}
\newcommand{\nn}{\mathbb{N}}

\newcommand\be{\begin{equation}}
\newcommand\ee{\end{equation}}
\newcommand\ba{\begin{array}}
\newcommand\ea{\end{array}}
\newcommand{\bean}{\begin{eqnarray*}}
\newcommand{\eean}{\end{eqnarray*}}

\newcommand{\ms}{\medskip}

\def \half{ \frac{1}{2}}

\def \Tr{ \mathop{\rm Tr}}

\def \val{ \mathop{\rm val}}

\def \Int{ \mathop{\rm int}}
\def \bd{ \mathop{\rm bd}}

\def\rar{\rightarrow}

\renewcommand{\bd}{\partial}

\def \y0{\bar Y_{0}}

\def \Arw{\mathop{\rm Arw}}

\def \ri{\mathop{\rm ri}}
\def \vec{\mathop{\rm vec}}
  \renewcommand\half{\mbox{$\frac{1}{2}$}}

\title{Refining the partition for multifold conic optimization problems}

\author{H\'ector Ram\'irez C.\thanks{Mathematical Engineering Department \& Center for  Mathematical Modeling (UMI 2807, CNRS), Universidad de Chile, Beauchef 851, Casilla 170-3, Santiago 3, Chile. E-mail: {\tt\small hramirez@dim.uchile.cl}}, \, Vera Roshchina\thanks{School of Mathematics and Statistics, UNSW Sydney, Kensington, 2052, NSW, Australia. While working on the paper this author was also affiliated with RMIT University (Melbourne, Australia) and Federation University Australia (Ballarat, Australia). E-mail: {\tt\small v.roshchina@unsw.edu.au}}}

\begin{document}
\maketitle

\begin{abstract}
In this paper we give a unified treatment of two different definitions of complementarity partition of multifold conic programs introduced independently in
[J. F. Bonnans and H. Ram\'irez C., Math. Program. 104 (2005), no. 2-3, Ser. B, 205--227]
for conic optimization problems, and in
 [J. Pe\~na and V. Roshchina, Math. Program. 142 (2013), no 1-2, Ser. A, 579--589] 
 for homogeneous feasibility problems. We show that both can be treated within the same unified geometric framework, and extend the latter notion to  optimization problems. We also show that the two partitions do not coincide, and their intersection gives a seven-set index partition.  Finally, we demonstrate that the partitions are preserved under the application of nonsingular linear transformations, and in particular that a standard conversion of a second-order cone program into a semidefinite programming problem preserves the partitions. 
\end{abstract}

\section{Introduction}\label{sec-introduction}

The complementarity partitions for linear programming and linear complementarity problems are well understood and are utilized heavily in the analysis of optimization problems as well as in numerical methods. In addition to strong duality (i.e., when both primal and dual problems have optimal solutions, and the values of their objective functions coincide), linear programming problems exhibit strict complementarity: in an appropriately formulated pair of primal and dual problems each primal variable has a dual counterpart, and there exists an optimal solution that is also a fully complementary pair, such that exactly one variable in each pair is positive and the other one is zero. The situation is much more complex for general linear conic problems. In a multifold system, where in place of the componentwise inequalities we have cone inclusions, it may happen that the relevant components of the primal and dual variables lie on the boundary of a corresponding cone. Moreover, there are several different definitions of complementarity that do not coincide and lead to different characterizations.

In some special cases, specifically when a homogeneous feasibility problem is considered, complementarity partition can be identified via an algorithm. The polyhedral (linear programming) case was considered in \cite{JavierNegar,Vavasis}, and it was shown in \cite{Terlaky} that it is possible to identify the partition of \cite{BonRam} via a variant of an interior point method for second-order cone programming. The complementarity partition is related to generalized condition measures for multifold conic systems, see \cite{CuckerPena}. Different approaches to complementarity partition for optimization problems  based on faces that contain primal and dual solutions can be found in  \cite{TuncelWolk,Yildirim}. These studies commonly treat the cone as a whole, treating partition as an intrinsic geometric notion. Our take on the problem is entirely different, because we impose a multifold structure of the cone explicitly and define the partition in terms of the block indices.

Recall that in the linear case complementarity partition has a simple and well-studied structure. Consider the primal-dual pair of linear programming problems
\begin{equation} \label{pb:PL}
\begin{split}
\min_{x\in \rr^n} \quad & c\transp x \notag \\
 \text{s.t.} \quad  & Ax \geq  b,   
\end{split} \quad \text{(LP)}\qquad \quad 
\begin{split}
\max_{y\in \rr^m} \quad & b \transp y \notag \\
 \text{s.t.} \quad  & A^\top y =  c,\\
 & y\geq 0.   
\end{split} \quad \text{(LD)}
\end{equation}
Here $A\in \R^{m\times n}$, $b\in \R^m$, $c\in \R^n$.
 The classical Goldman-Tucker theorem \cite{GoldmanTucker} (also see \cite[Proposition 17.16]{BGLSa} for modern treatment) states that when either one of the problems has a finite optimal value, there exists a unique partition $(B,N)$ of the index set $J=\{1,\dots, m\}$ that is of maximal complementarity. In other words, there is a  pair  $(\bar x, \bar y)$ of optimal solutions to (LP) and (LD) such that for $\bar s= A\bar x-b$ and $\bar y$ we have $\bar s_B>0$, $\bar y_N>0$, and moreover for any primal-dual optimal solution $(x,y)$, with $s: = A x-b $, it always holds that $ y_B = 0$, $ s_N = 0$ (here we have used the notation standard in linear programming literature: by $v_I$ we denote the vector of entries of $v = (v_1,\dots, v_m)\in \R^m$ with indices from a subset $I\subseteq J$ of $J = \{1,2,\dots, m\}$).  

The structure of optimal partition of monotone linear complementarity problems is also  well-known. Indeed, let $Q$ and $R$ be $n\times n$ matrices, and $h\in \R^n$. Recall that the linear complementarity problem 
$$
\begin{cases}
x\transp s = 0,\\
Qx+Rs = h,\\
x\geq 0, \quad s\geq 0,
\end{cases}
$$
is called monotone if $Qu+Rv = 0$ yields $u \transp v\geq 0$. For this type of problems the optimal partition is of the form $(B,N,T)$, where $T$ corresponds to the subset of indices for which both $s$ and $x$ components are zero for every solution; see, e.g., \cite{IPRT00,KMY89,MW94}.

\if{The optimal partition features in the study of monotone linear complementarity problems. Let $Q$ and $R$ be $n\times n$ matrices, and let $h\in \R^n$. The linear complementarity problem 
$$
\begin{cases}
x\transp s = 0,\\
Qx+Rs = h,\\
x\geq 0, \quad s\geq 0,
\end{cases}
$$
is called monotone if $Qu+Rv = 0$ yields $u \transp v\geq 0$. For monotone linear complementarity problems the optimal partition is of the form $(B,N,T)$, where $T$ corresponds to the subset of indices for which both $s$ and $x$ components are zero for every solution. 
}\fi

This linear complementarity partition is reminiscent of the situation encountered in multifold conic systems. When the system of linear inequalities is replaced by a product of general closed convex cones (as in the case of second-order cone programming for instance), the partition of indices becomes more complex since the component variables may end up on the boundary of the cones. 

In \cite{BonRam} a four-set partition was introduced for general multifold conic systems, based on regularity conditions involving normal cones. In \cite{PR2012} a refined six-set partition was suggested for homogeneous feasibility problems and a geometric characterization of such partition was obtained. Our goal is to generalize the latter partition to multifold conic optimization problems and compare the two partitions. 

We focus on the case when strong duality holds for the primal-dual pair of conic optimization problems. We note that this is a minimal assumption required for the notion of complementarity to make sense: the optimal solutions must exist, and this in particular can be guaranteed by the presence of Slater points in both primal and dual problems. In principle, the assumption of strong duality can always be guaranteed by performing a facial reduction preprocessing step if necessary; see \cite{FRedPolyh,PatakiStrong,WakiMuramatsu}. Another possible approach is to use Ramana Duals \cite{Ramana,StrongDuality}. Both approaches need to assume that the optimal value of the primal is attained. As pointed out by a reviewer, another way to remedy the failure of strong duality, and thus to ensure the existence of the optimal partition, is to cast the primal problem into a self-dual embedding format, for which the Slater condition always holds, see e.g., \cite{KRT1,KRT2}. Then, the optimal partition of the original problem, if exists, can be recovered from the optimal partition of the embedding problem.


We introduce the two types of partitions in Section~\ref{sec:partitions}, show the relations between them in Theorem~\ref{thm:RinC}, treat the special case of second-order cones in Proposition~\ref{lem:partsoc}  and provide an illustrative example. In Section~\ref{sec:hom}, we link the geometric relations of  \cite{PR2012} to the four-set partition of \cite{BonRam} in Theorem~\ref{thm:RDualChar}. We conclude the paper with a study of the relations between partitions of reformulated (lifted) problems in Section~\ref{sec-dual-socp-b}, strengthening some results of \cite{BonRam}.

Throughout the paper, we work in the finite-dimensional Euclidean setting of a real vector space $\R^n$ endowed with the standard Euclidean norm $\|x\| = \sqrt{x\transp x}$.


\section{Partitions for multifold conic optimization problems}\label{sec:partitions}

Consider a general linear conic optimization
problem with constraints in product form, i.e.,
\begin{align}
\label{pb:cop}
 \min_{x\in \rr^n} \quad &  c\transp x \notag \\
 \text{s.t.} \quad  &    A^j x - b^j  \in K_j \quad  \forall j\in J,\tag{P}
\end{align}
where $K_j$ is a closed convex regular cone in $\rr^{q_j}$, $q_j \in \nn$ for every $j\in J$, and $J = \{1,\dots, r\}$ is a finite index set (a regular cone has a nonempty interior and is pointed, i.e. it does not contain lines).
We set $K :=K_1\times\cdots\times K_r$, and define
$A=(A^1;\cdots;A^r)$ as the matrix
whose rows are those of $A^1$ to $A^r$, and $b := \vec(b^1,\ldots,b^r)$
so that \eqref{pb:cop} is equivalent to
$$
\min_{x\in \rr^n} \{ c\transp x; \; Ax-b \in K\}.
$$
The dual problem is 
\begin{align}
\label{pb:dcop}
\max_{y^1,\ldots,y^r} \quad & 
 \sum_{j=1}^r (b^j)\transp y^j\notag \\ 
\text{s.t.}\quad &  \sum_{j=1}^r (A^j)^\top y^j = c,\tag{D}\\
& y^j \in K_j^+ \quad \forall j\in J,\notag
\end{align}
where the (positive) polar of a set $C\subset \rr^m$
is defined as
$C^+ :=\{ y\in \rr^{m}; \; y\transp z \geq 0, \;
\forall  z \in C\}$.
Generally speaking, such primal-dual pairs of linear conic problems may have a nonzero duality gap (see \cite[Section~11.6]{Guler} for a detailed discussion and \cite[Section~3.2]{Renegar} for examples). The Slater condition (i.e., the existence of $\hat x$ such that $A^j \hat x - b^j  \in \Int K_j$,  $\forall j\in J$, for its primal version, or an analogous condition for the dual) guarantees the absence of the duality gap. Together with the bounded objective for either \eqref{pb:cop} or \eqref{pb:dcop} this also yields the existence of the optimal solution for the relevant primal or dual counterpart. Also note that zero duality gap alone does not guarantee the existence of optimal solutions for both primal and dual problems (see, e.g., \cite{K02}). So, we make the following standing assumption throughout the paper:
\ms

{\bf Standing assumption}: \emph{strong duality holds for the pair \eqref{pb:cop}-\eqref{pb:dcop}}, i.e, the optimal solution sets of \eqref{pb:cop} and \eqref{pb:dcop}  are nonempty, and the duality gap is zero. 
\ms

A pair $(x,y)$ of optimal solutions to primal and dual problems is characterized by the
complementarity system
\begin{align}
\label{pb:copos}
A^j x - b^j & \in K_j, \notag \\
 y^j & \in K_j^+,\notag \\
(y^j) \transp (A^j x - b^j)& =0 \quad  \forall  j\in  J,\tag{C}\\ 
A^\top y & = c.\notag
\end{align}
We denote by $S\eqref{pb:copos}$ the set of solutions for relations \eqref{pb:copos}. Observe that $S\eqref{pb:copos}$ is nonempty if and only if strong duality holds (which is our standing assumption). We also denote by $S\eqref{pb:cop}$ and $S\eqref{pb:dcop}$  the set of solutions to the problems \eqref{pb:cop} and \eqref{pb:dcop}, respectively, and by $F\eqref{pb:cop}$ and $F\eqref{pb:dcop}$ their respective feasible sets. Notice that strong duality implies the equality $S\eqref{pb:copos}=S\eqref{pb:cop}\times S\eqref{pb:dcop}$.


We say that {\em strict primal (resp. dual) feasibility} holds for
$j \in J$ if there exists
$x \in F\eqref{pb:cop}$ such that $A^j x-b^j \in \intt K_j$
(resp. $y\in F\eqref{pb:dcop}$ such that $y^j \in \intt K_j^+$). 

\if{

\begin{lema}
\label{strict-local}
Let $j$ be strictly primal (resp. dual) feasible.
Then the set
$\{y^j; y\in S\eqref{pb:dcop} \}$
(resp. $\{A^j x-b^j; x\in S\eqref{pb:cop} \}$)
is bounded.
\end{lema}

\begin{proof}
If $j$ is strictly primal feasible, there exists $x\in F\eqref{pb:cop}$ and
$\eps >0$ such that $s=A x-b$ satisfies $s^j + \eps B \subset
K_j$, or equivalently $\eps B \subset s^j - K_j$. Let $y \in
S\eqref{pb:dcop}$. Since $y^j\in K_j^+$, it follows that $\eps \|y^j\|
\leq y^j \transp s^j$. Using also $y^{j'} \transp s^{j'} \geq 0$, for
all $j'$, we get
$$
0= x\transp ( c-A^\top y) = c\transp x - y\transp Ax
  = c\transp x - b\transp y - y\transp s
  \leq c\transp x - b\transp y - \eps \|y^j\|.
$$
In other words,
$\eps \|y^j\| \leq c\transp x - b\transp y = c\transp x - \val\eqref{pb:dcop}$,
which gives the desired estimate.
The proof for the dual statement is similar.
\end{proof}

Note that the hypothesis of the above lemma does not imply the absence
of a duality gap.

}\fi

One says
(e.g., \cite[Def. 4.74]{BSbook})
that the {\em strict complementarity condition}
holds for problem \eqref{pb:cop} if there exists a pair
$(x,y)$ solution of the optimality system \eqref{pb:copos}, such that
$-y \in \ri N_K(Ax-b)$, where $N_K$ is the standard normal cone of convex
analysis (see \cite[Section~A.5.2]{JBHULem}). Since $K$ is a closed convex cone, we have for $s\in K$ that
\begin{equation}
\label{formul-nk}
N_K(s) = \left\{
\begin{array}{cc}
(-K^+) \cap s^\perp, & s\in \partial K,\\
\emptyset, & s\notin K,\\
\{0\}, & s\in \intt K,\\
\end{array}
\right.
\end{equation}
where $s^\bot$ denotes the orthogonal complement to the linear span of $s$.

For problems with constraints in the product form as in \eqref{pb:cop}, we introduce the notion of {\em componentwise strict complementarity}, which means that for each component $j$ there exists
a pair $(x,y) \in S\eqref{pb:copos}$, such that $-y^j \in \ri
N_{K_j}(A^jx-b^j)$. As we will see later, the two notions are equivalent.

In \cite{BonRam} the notion of optimal partition, well-known for linear programming and
monotone linear complementarity problems
(see, e.g., \cite[Section 18.2.4]{BGLSa}), was extended to our abstract framework.
Let 
\begin{align*}
B   &= \{j\,|\, \exists (x,y) \in S\eqref{pb:copos}\; \text{s.t.}\; A^j x-b^j \in \intt K_j\}\\
N   &= \{j\,|\, \exists (x,y) \in S\eqref{pb:copos}\; \text{s.t.}\; y^j \in \intt K_j^+\}\\
R^0 &= \{j|\, \exists (x,y) \in S\eqref{pb:copos}\; \text{s.t.}\; -y^j \in  \ri N_{K_j}(A^j x - b^j)\}.
\end{align*}
Note that $B\cup N\subseteq R^{0}$. 
Define $R:= R^0\setminus (B\cup N)$, $T:=J\setminus R^0$.
It was shown in \cite[Lemma 3]{BonRam} that if $S\eqref{pb:copos}$ is not empty, the partition $(B,N,R,T)$ is a disjoint partition of the index set $J$.



\begin{figure}[ht]
	{\centering \includegraphics{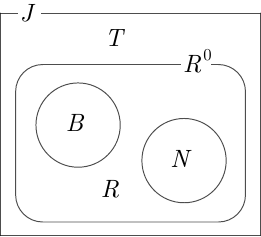}\\}
	\caption{The four-partition $(B,N,R,T)$.}
	\label{fig:BTNR}
\end{figure}

\begin{defi}
Any pair $(x,y) \in S\eqref{pb:copos}$ satisfying the relations below is
said to be of maximal complementarity:
\begin{equation}
\label{pb:copri}
\left\{ \ba{lll}
{\rm (i)} \;
A^j x-b^j \in \intt K_j, \; \;\forall j\in B,\\
{\rm (ii)}  \;
y^j \in \intt K_j^+, \; \;\forall j\in N,
\\
{\rm (iii)}  \;
-y^j \in \ri N_{K_j}(A^j x-b^j), \; \;\forall j\in R.
\ea \right.
\end{equation}
\end{defi}

For each $j\in B\cup N\cup R$ let $(x(j),y(j)) \in S\eqref{pb:copos}$ be such that
$$ 
\begin{cases}
A^j x^j(j)-b^j \in \intt K_j & \text{if}\quad j \in B,\\
y^j(j) \in \intt K_j & \text{if}\quad j \in N,\\
-y^j(j) \in \ri N_{K_j}(A^j x(j)-b^j) & \text{if}\quad j \in R,
\end{cases}
$$
where $x^j(j)$ and $y^j(j)$ stand for the block $j$ of vectors $x(j)$ and $y(j)$, respectively.

We define
\begin{equation}\label{eq:maxcomconstr}
\xh := (|B|+|R|)^{-1}\sum_{j\in B\cup R} x(j); \quad
\yh := (|N|+|R|))^{-1}\sum_{j\in N\cup R} y(j).
\end{equation}

\if{

Let us state some properties of the set of
maximal complementarity solutions.
We need a preliminary lemma.

\begin{lema}
\label{mcopos.l4a}
Let $K$ be a closed convex cone.
Let $s^i \in K$, for $i=1,2$,
$- y^1 \in N_K(s^1)$, and $- y^2 \in \ri N_K(s^2)$.
Given $\alpha \in ]0,1[$, set
$(s,y) := \alpha (s^1,y^1) + (1-\alpha) (s^2,y^2)$.
If $-y\in N_K(s)$, then $-y\in \ri N_K(s)$.
\end{lema}

\begin{proof}
Since $-N_K(s)=K^+\cap s^\perp$, we have that
$-y \in \ri N_K(s)$ iff, for all $z\in N_K(s)$,
$y \pm \eps z \in K^+$ for small enough $\eps>0$.
As $K^+$ is a cone, $y + \eps z \in K^+$
always holds. Therefore we have to prove that for
$z\in N_K(s)$,
$y - \eps z \in K^+$ for small enough $\eps>0$.
Using
$N_K(s) = N_K(s^1) \cap N_K(s^2)$,
obtain $z \in N_K(s^2)$, and hence,
$y^2 - \eps' z \in K^+$ for some $\eps'>0$.
Let $\eps:= (1-\alpha)\eps'$. Then
$
y - \eps z = \alpha y^1 + (1-\alpha)(y^2 - \eps' z)
$
belongs to $K^+$.
The conclusion follows.
\end{proof}
}\fi

It was shown in \cite[Lemma 7]{BonRam} that the pair $(\xh,\yh)$ defined in \eqref{eq:maxcomconstr} is of maximal complementarity, and that any pair $(x,y) \in \ri S\eqref{pb:cop} \times \ri S\eqref{pb:dcop}$ is of maximal complementarity. An immediate consequence of this result is that componentwise strict complementarity yields strict complementarity (observing that $\ri N_K(Ax-b) = \ri N_{K_1}(A^1 x- b^1)\times \ri N_{K_2}(A^2 x- b^2)\times \cdots \times\ri N_{K_r}(A^r x- b^r)$), therefore the two notions are equivalent.

In a separate development \cite{PR2012}, a different idea was used to define a partition for a homogeneous feasibility problem. In the case when $b=0$ and $c=0$ the pair \eqref{pb:cop}-\eqref{pb:dcop} is reduced to a pair of homogeneous feasibility problems \eqref{pb:fcop}-\eqref{pb:dfcop} below. In this case the primal problem consists of finding an $x$ such that  
\begin{equation}
\label{pb:fcop}\tag{FP}
 A^j x  \in K_j \quad \forall j\in J,
\end{equation}
and the dual problem consists of finding $y$ satisfying
\begin{equation}
\label{pb:dfcop}\tag{FD}
\sum_{j=1}^r (A^j)^\top y^j = 0, \qquad 
y^j \in K_j^+ \quad \forall  j\in J.
\end{equation}
Either one of the problems has a strictly interior solution if and only if the other one only has a zero solution; the most interesting situation in terms of the partition indices occurs when both problems have boundary solutions. We can assign an index $j\in J$  to the set $B^0$ if all solutions to the dual problem \eqref{pb:dfcop} have zero $j$-th component, and to $N^0$ if for every primal solution to the problem \eqref{pb:fcop} the $j$-th component is zero. The remaining two partition sets are the intersection $O = B^0\cap N^0$ and the complement $C = J\setminus B^0\cup N^0$. In what follows we extend this definition to optimization problems. Define the index subsets $B^0$  and $N^0$:
\begin{align*}
& B^0 = \{j\,|\, \forall (x,y) \in S\eqref{pb:copos}\quad   y^j = 0\},\\ 
& N^0 = \{j\,|\, \forall (x,y) \in S\eqref{pb:copos}\quad  A^j x-b^j = 0\},
\end{align*}
and define the following four sets:
$$
B' = B^0\setminus(N^0\cup B), \quad  N' = N^0\setminus(B^0\cup N), \quad O = B^0\cap N^0, \quad C = J\setminus(B^0\cup N^0).
$$
It follows immediately from the complementarity conditions of \eqref{pb:copos} that $B\subset B^0$ and $N\subset N^0$, and that $B'\cap N' = \emptyset $. Therefore, it is not difficult to observe that the sets $B,N,B',N',O,C$ form a disjoint partition of the index set $J$. To avoid confusion, we will refer to this partition as the six-partition, whereas the four-set partition introduced before will be referred to as the four-partition.

\begin{figure}[ht]
	{\centering \includegraphics{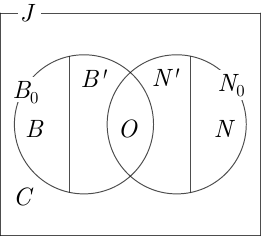}\\}
	\caption{The six-partition.}
	\label{fig:six-partition}
\end{figure}

\begin{teo}\label{thm:RinC} The following relations between the four- and six-partitions hold.
\begin{equation}\label{eq:RinC}
T \supset B'\cup N' \cup O = (N^0\cup B^0)\setminus (N\cup B);\quad R \subset C=J\setminus(B^0\cup N^0);
\end{equation}
\end{teo}
\begin{proof} The first relation follows directly from the fact that the six-partition is indeed a partition of set $J$ and $R\subset C$. Thus, it only remains to show that $R\subset C$. We argue by contradiction and suppose the existence of an index $j^*\in R\setminus C$. Since $C=J\setminus(B^0\cup N^0)$, $j^* \in B^0\cup N^0$.

Since $j^*\in R$, there exists a solution  $(x(j^*),y(j^*)) \in S\eqref{pb:copos}$ with 
\begin{equation}\label{eq:pf01}
- y^{j^*}(j^*) \in \ri N_{K_{j^*}}(A^{j^*} x(j^*) - b^{j^*}).
\end{equation} 

First assume that $j^* \in B^0\cap R$. This yields $y_{j^*}(j^*)=0$, and hence 
$$
0\in \ri N_{K_{j^*}}(A^{j^*} x(j^*) - b^{j^*}).
$$
This is only possible when either $N_{K_{j^*}}(A^{j^*} x(j^*) - b^{j^*})=\{0\}$ or $N_{K_{j^*}}(A^{j^*} x(j^*) - b^{j^*})$ contains lines. In the first case $A^{j^*} x(j^*) - b^{j^*}\in \intt K_{j^*}$, and $j^*\in B$, and that contradicts the definition of $R$. The second case is impossible by our assumption that all cones $K_j$, $j\in J$ have a nonempty interior.

Now assume that $j^* \in N^0\cap R$. We have $A^{j^*} x(j^*) - b^{j^*}=0$, hence, \eqref{eq:pf01} yields
$$
-y^{j^*}(j^*)\in \ri N_{K_{j^*}}(A^{j^*} x(j^*) - b^{j^*}) = \ri N_{K_{j^*}}(0) = \ri (-K^+_{j^*}) = -\intt K^+_{j^*},
$$ 
which yields $j^*\in N$, and contradicts the fact that $j^*$ is in $R$. We therefore conclude that 
\begin{equation}\label{eq:pf2}
R\subset C = J\setminus(B^0\cup N^0).
\end{equation}
\end{proof}

We demonstrate the relations of Theorem~\ref{thm:RinC} using a Venn diagram in Fig.~\ref{fig:venn}.
\begin{figure}[ht]
{\centering \includegraphics{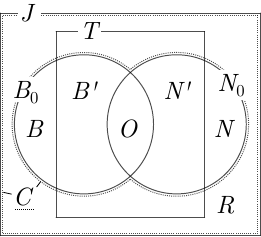}\\}
\caption{The relations between index sets (see Theorem~\ref{thm:RinC}).}
\label{fig:venn}
\end{figure}

%
%

We next show that in the case of second-order cone programming the complementarity partition has special properties. Recall that a second-order cone is a product of Lorentz cones (closed convex cones  $L_n$ in $\rr^{n+1}$ for some $n\geq 1$ defined as follows,
$$
L_n = \{x=(x_0,\bar x)\in \rr\times \rr^n\,|\, x_0\geq \|\bar x\|\},
$$
i.e., such cones are the epigraphs of the Euclidean norm).
For consistency, we also let $L_0:= \R_+ = \{x\in \R\,|\, x\geq 0\}$. Second-order cone programs are of type \eqref{pb:cop} and \eqref{pb:dcop}, where each cone $K_j$, for $j\in J$, is a Lorentz cone. Notice that linear programming case can be cast as a special case of second-order cone programming when $K_j=L_0$ for all $j\in J$.
\begin{prop}\label{lem:partsoc} If for every $j\in J$ the cone $K_j$ is a Lorentz cone, then 
\begin{equation}\label{eq:RisC}
 R = C=J\setminus(B^0\cup N^0).
\end{equation} 
\end{prop}
\begin{proof}
Let $j\in C$. Then, there exists a solution  $(x(j),y(j)) \in S\eqref{pb:copos}$ with both 
$y_j(j)$ and $s = A^{j} x(j) - b^{j}$ nonzero. 
Since $y_j(j)\perp s$, we have $y_j(j)\in (K^+_j)\cap s^\perp = -N_{K_j}(s)$. Observe that for a second-order cone the 
normal cone  $N_{K_j}(s)$ is one-dimensional provided $s\neq 0$, hence, $\ri N_{K_j}(s) = N_{K_j}(s)\setminus \{0\}\ni y_j(j)$, and
therefore $j\in R$. We have then shown that $C\subset R$. The reverse inclusion follows from Theorem~\ref{thm:RinC}.
\end{proof}

We next show a simple example of mixed polyhedral and Lorentz cones in which it can happen that $R\neq C$. 
\begin{ej} We consider a feasibility problem given by $b=0$, $c=0$ and
$$
A^\top  = \left[
\begin{array}{cccccccc}
0 &   1 & 0 & 0& 0& 1 &  0 & 0 \\
0 & 0 & 1 & 0 & -1 &  1 &  1 & 1 \\
1 & 0 & 0 & 1 &   0 &  0 &  0 & 1 
\end{array}
\right].
$$

First, we consider 
$$
K = L_3\times \rr_+^1\times \rr_+^1\times\rr_+^1\times \rr_+^1.
$$
Solving the problem analytically, we obtain all possible solutions as the parametric family
$$
x = (0,0,\alpha)^\top , \;  Ax = (\alpha, 0 , 0, \alpha, 0, 0, 0, \alpha)^\top ,\; y = (\beta, 0, 0, -\beta,\gamma, 0, \gamma, 0)^\top , 
$$
where $\alpha, \beta,\gamma\geq 0.$ We thus obtain that $1\in C$, $2,4\in N$, $3\in O$, $5\in B$, and $R = \{1\}$, since the first cone is a Lorentz cone. Hence, as expected from Proposition \ref{lem:partsoc}, in this case $R=C$.
However, if we let $K = L_3\times \rr^1_+\times \rr_+^3$, we directly deduce that $1,3 \in C$, $2\in N$. Now observe that 
$$
N_{K_3}(A_3x) = N_{\rr^3_+}(0,0,\alpha) = -\rr^2_+\times\{0\},
$$ 
provided $\alpha >0$, and hence
$$
\ri N_{K_3}(A_3x) = -\intt \rr_+^2 \times \{0\}.
$$
Hence, since the only possible solution is $y_3 = (0,\gamma,0)$, with $\gamma\geq 0$, it clearly follows that $-y_3\notin \ri N_{K_3}(A_3x)$. Hence, $3\notin R$.
\end{ej}

\if{

{\tt MAKE NOTATION CONSISTENT FOR THE SOCs}

{\tt Ve It is interesting to find the conditions under which $C=R$

ONE ANSWER IS WHEN WE CAN ENSURE THAT $S\eqref{pb:copos}=S\eqref{pb:cop}\times S\eqref{pb:dcop}$}.

}\fi

\section{Homogeneous feasibility problems}\label{sec:hom}

In this section we focus on homogeneous feasibility problems \eqref{pb:cop}--\eqref{pb:dcop} introduced earlier in Section~\ref{sec:partitions}. Observe that in this case the strong duality property coincides with solvability of problems (FP) and (FD). The relevant complementarity conditions can be obtained from \eqref{pb:copos},
\begin{align}
\label{pb:fcopos}
A^j x & \in K_j, \notag \\
 y^j & \in K_j^+,\notag \\
(y^j) \transp A^j x & =0 \quad  \forall  j\in  J,\tag{FC}\\ 
A^\top y & = 0.\notag
\end{align}

In \cite{PR2012} a dual characterization of the homogeneous six-partition was obtained.
\begin{teo}[{\cite[Theorem 1]{PR2012}}]\label{thm:BNDualChar} The sets $B$, $N$, $B^0$ and $N^0$ can be described as follows.
\begin{align*}
& B = \{j\in J \,|\,  A_j^\top ( K_j^+\setminus \{0\}) \cap \Lin (\overline{A^\top  K^+})= \emptyset\},\\
& N = \{j\in J \,|\, \ri A_j^\top  K_j^+ \cap \Lin (A^\top  K^+)\neq \emptyset\},\\
& B^0 = \{j\in J \,|\,  A_j^\top ( K_j^+\setminus \{0\}) \cap \Lin (A^\top  K^+)= \emptyset\},\\
& N^0 = \{j\in J \,|\, \ri A_j^\top  K_j^+ \cap \Lin (\overline{A^\top  K^+})\neq \emptyset\}.
\end{align*}
Here by $\Lin C$ we denote the lineality space of a convex set $C$.
\end{teo}

It appears that in the case of feasibility problem the set $R$ has a similar dual property. The next result extends the relations of Theorem~\ref{thm:BNDualChar}.

\begin{teo}\label{thm:RDualChar} In the case of a homogeneous feasibility problem 
$$
R^0 \subseteq \tilde R  := \{j\in J \,|\,   A_j^\top  K_j^+\cap \Lin ({A^\top  K^+}) =  A_j^\top  K_j^+\cap \Lin (\overline{A^\top  K^+})\}.
$$
\end{teo}
\begin{proof} We argue by contradiction and assume that, for some pair of problems \eqref{pb:fcop}-\eqref{pb:dfcop}, there exists $j\in R^0\setminus \tilde R$.
	
Consider the linear subspaces 
$$
 L = \{y\,|\, A_j^\top  y \in \Lin (A^\top  K^+)\}, \qquad \bar L = \{y\,|\, A_j^\top  y \in \Lin (\overline{A^\top  K^+})\},
$$
and let 
$$
F = L\cap K^+_j,\quad  \bar F = \bar L\cap K^+_j.
$$
Here observe that $F\subsetneq \bar F$: indeed, since $j\notin \tilde R$, the sets 
$$
 A_j^\top  K_j^+\cap \Lin ({A^\top  K^+}) \; \text{ and } \;   A_j^\top K_j^+\cap \Lin (\overline{A^\top  K^+})
$$
are different, and hence there exists some $p\in A^\top  K_j^+ $ such that $p\in \Lin (\overline{A^\top  K^+})\setminus \Lin ({A^\top  K^+}) $. Subsequently there exists $u\in K_j^+$ such that $A\transp u = p$. For every such $u$ we have $u\in \bar L \setminus L$, yielding $F\neq \bar F$.

Pick an arbitrary solution $x$ to \eqref{pb:fcop} and any $v\in \bar F$. Since $v\in \bar F$, we have 
$A_j^\top  v\in \Lin(\overline{A^\top  K^+})$, hence, $-A_j^\top  v\in \Lin(\overline{A^\top  K^+})\subseteq \overline{A^\top  K^+}$. 
Then there exist sequences $\{v^+_k\}$, $\{v_k^-\}$ such that $v_k^-,v_k^+\in K^+$, $A^\top  v_k^-, A^\top  v_k^+\in A^\top  K^+$, $A^\top  v_k^\pm \to \pm A^\top _j v$. Since $A^\top  v_k^-,A^\top  v_k^+ \in A^\top  K^+$, and since $x$ is a solution (i.e. $Ax\in K$), we have for some $y\in K^+$, $A\transp v_k^- = A\transp y$
$$
 (Ax)^\top v_k^{-} = x \transp A^\top  v_k^- = x A^\top y =  (Ax)^\top y \geq 0,  
$$
and analogously $0\leq  (Ax)^\top v_k^+$ for all $k$. Hence, passing to the limit, we have $(A_j x)\transp v = 0$. Therefore, 
$\bar F \subseteq (A_j x)^\perp$. By definition, $\bar F \subseteq K^+_j$, hence,
$$
\bar F \subseteq K^+_j \cap (A_jx)^\perp.
$$

Since $j\in R^0$, there exists a pair of feasible solutions $(x,y)$ to \eqref{pb:fcop}-\eqref{pb:dfcop}  such that 
\begin{equation}\label{eq:yinriN01}
y_j \in -\ri N_{K_j}(A_j x), \; \text{ where } N_{K_j}(A_j x)  = -K^+_j \cap (A_jx)^\perp.
\end{equation}
Let $\bar y \in \rr^n$ be such that $\bar y_i = 0 $ for all $i\neq j$, and $\bar y_j= y_j$. Observe that $\bar y\in K^+$, hence, $A^\top _j y_j = A^\top  \bar y \in A^\top K^+$, and since $A^\top  y = 0$,  we have $-A^\top_j y_j \in A^\top K^+$, hence, $A^\top _j y_j \in \Lin(AK^+)$, and therefore $y_j\in F$.  

Since $y_j \in F\subsetneq \bar F \subseteq  K^+_j \cap (A_jx)^\perp$, $F = \bar F \cap L$, and all three sets $F$, $\bar F$ and $K^+_j \cap (A_jx)^\perp$ are cones, their affine hulls are linear subspaces, moreover, we have the following representations
$$
\Aff F = L_1, \quad \Aff \bar F = L_2, \quad \Aff (K^+_j \cap (A_jx)^\perp) = L_3,
$$
where $L_1, L_2, L_3$ are linear subspaces, $L_1 \subsetneq L_2\subseteq  L_3$. Observe that $L_1 \subsetneq L_2$ follows from the fact that $ F = \bar F\cap L$, where $L$ is a linear subspace and $F\neq \bar F$. Also note that since $F = L\cap \bar F$, and $\bar F = \bar L \cap (K^+\cap (A_j x)^\perp)$, we have $F = \bar F \cap L_1$, $\bar F = L_2\cap (K^+\cap (A_j x)^\perp)$.

From \eqref{eq:yinriN01} we deduce that there exists $\varepsilon>0$ such that $(y_j+ B_\varepsilon)\cap L_3 \subseteq K^+_j \cap (A_jx)^\perp$. Observe that since $\Aff \bar F = L_1\oplus(L_2\cap L_1^\perp)$, there exists a unit vector $u$ in $L_1^\perp\cap L_2$ such that  $y_j +\varepsilon u\in \bar F\subseteq L_2\subseteq L_3$, therefore $y_j-\varepsilon u \in L_3$, and since $\pm \varepsilon u \in B_\varepsilon$, we have $y_j\pm \varepsilon u \in L_2\cap( K^+_j \cap (A_jx)^\perp)$, hence, $y_j\pm \varepsilon u\in \bar F$, as $\bar F = \bar L_2\cap (K^+ \cap (A_jx)^\perp )$. 

Therefore, $A^\top _j(y_j\pm \varepsilon u) \in A^\top K^+$. Since $A_j^\top y_j \in \Lin(A\transp K^+)$, we have $-A^\top _j y_j \in A^\top K^+$,  and therefore  we deduce that $-A_j^\top( y_j\pm \varepsilon u) \in A\transp K^+$, and hence $A_j^\top (y_j\pm u)\in \Lin(A\transp K)$, which means that $y_j\pm \varepsilon u \in F$, which contradicts the construction of $u$. Therefore, our assumption that there exists $j\in R^0\setminus \tilde R$ is wrong, and hence $R^0\subseteq \tilde R$.

\end{proof}

In the following example we show that the reverse inclusion in $R^0 \subseteq \tilde R$ does not always hold.
\begin{ej} Consider a feasibility problem given by $b=0$, $c=0$, 
$$
A^\top  = \left[
\begin{array}{cccccccc}
0 &   1 & 0 & 0& 0& 1& 0 &  0  \\
0 & 0 & 1 & 0 & -1 &0 &  1 &  0  \\
1 & 0 & 0 & 1 &  0  &0 &  0 &  1 
\end{array}
\right]
$$
\if{
$$
A^\top  = \left[
\begin{array}{ccc}
0 &   0 & 1\\
1 & 0& 0 \\
0& 1 &  0 \\
0 & 0 & 1\\
1 & 0& 0 \\
0& 1 &  0 \\
0 & 0 & 1\\
0 &  -1 & 0 \end{array}
\right]
$$
}\fi
and 
$$
K = L_3\times \rr_+^1\times \rr_+^3.
$$

Solving the problem analytically, we obtain all possible solutions as the parametric family
$$
x = (0,0,\alpha)^\top , \;  Ax = (\alpha, 0 , 0, \alpha, 0, 0, 0,  \alpha)^\top ,\; y = (\beta, 0, 0, -\beta, \gamma, 0,\gamma, 0)^\top , 
$$
where $\alpha, \beta,\gamma\geq 0.$ We thus obtain that $1\in R$, $2\in N$, and $3\in T$. Consequently,  $R^0=\{1,2\}$. However, for this concrete example it is possible to prove that $A^\top K^+$ is closed. Then, from the definition of  $\tilde R$, it follows that $\tilde R=\{1,2,3\}$, which thus turns to be strictly larger than $R^0$. Indeed, let us denote by $u\in L_3$,  $w\in \rr_+^1$ and $v\in \rr_+^3$ the components of $y \in K$. Hence,
\begin{align*}
A^\top K^+ 
& =\{ (u_1+v_1, u_2+v_2 -w, u_0 + u_3 + v_3)^\top : u\in L_3 , \, w\geq 0,\, v_i\geq 0 \, \forall i=1,2,3\} \\
& = \{ (x_1,x_2,x_3)^\top : x_2 \in \R, \, x_3\geq 0, \, x_1 + x_3\geq 0 \}.
\end{align*}
In the computation above we have used that: i) $v_2-w$ can be any real number when $v_2$ and $w$ are nonnegative; 
ii) the third component $x_3 = u_0 + u_3 + v_3$ is nonnegative because $u\in L_3$ and $v_3\geq 0$, and 
iii) the first component $x_1=u_1+v_1$ can be negative because of $u_1$. However, since $u\in L_3$, $u_0$ is always greater or equal than $-u_1$. This necessarily implies that $x_1 + x_3\geq 0$.
We thus conclude that $\tilde R=\{1,2,3\}$, and consequently $R^{0}\neq \tilde R$. 
\end{ej}

\section{Equivalent optimization problems}\label{sec-dual-socp-b}

We now introduce another problem related to \eqref{pb:cop}.
Let $\calk= \calk_1 \times \cdots \times \calk_r$ be another finite family of
regular closed convex cones in $\rr^{r_j}$, $j\in J$, and $M^j$ be
$r_j\times q_j$ matrices such that
\begin{equation}
\label{copos1}
s^j\in K_j \;\; \text{ iff } \;\; M^j s^j \in \calk_j, \; j=1,\ldots,J.
\end{equation}
The latter is inspired by the relation between second-order cones and the cone of the positive semidefinite matrices (denoted by $\Sym^{n+1}_+$). Indeed, consider, for some $n\geq 1$, the cones $L_n$ and $\Sym^{n+1}_+$. For $s=(s_0,\sb)\in \rr\times\rr^n$ it is easy to check that 
\begin{equation}
\label{eq:equivArw} s\in L_n  \;\;\text{ iff } \;\; 
\Arw(s) := 
\begin{pmatrix} s_0 & \sb^\top
                   \\ \sb & s_{0} I_n
\end{pmatrix} \in \SS^{n+1}_+.
\end{equation}
Linear function $\Arw(\cdot)$ is the  \emph{arrow function} and it has several useful properties that will be exploited in what follows.

Let $M = (M^1;\cdots;M^r)$ be the matrix whose rows are those of
$M^j$. Then, \eqref{pb:cop} is equivalent to the linear conic problem
\if {
\begin{equation}
\label{pb:mcop}
\Min_{x\in \rr^n}
 c\transp x\, ;\, M(A x - b )\in \calk,
\end{equation}
or equivalently
}\fi
\begin{equation}
\label{pb:mcopa}\tag{MP}
\Min_{x\in \rr^n}
 c\transp x\, \quad \text{ s.t. } M^j(A^j x - b^j )\in \calk_j, \; j\in J,
\end{equation}
whose dual is
\begin{equation}
\label{pb:dmcop}\tag{MD}
\Max_{z\in \calk^+}
\sum_{j=1}^r (b^j)\transp (M^j)^\top z^j \quad \text{ s.t. } 
\sum_{j=1}^r (A^j)^\top (M^j)^\top z^j = c; \; \; z^j\in \calk_j^+,
\; j\in J.
\end{equation}
If the strong duality holds for this problem, an optimal pair
$(x,y)$ of the primal and dual problems is characterized by the
optimality system
\begin{equation}
\label{pb:mcopos}\tag{MC}
\left\{ \ba{lll}
 M^j(A^j x - b^j) \in \calk_j,  \; \;
z^j\in \calk_j^+, \;\; (z^j) \transp M^j (A^j x - b^j)=0,
\; j\in J;
\\
\sum_{j=1}^r (A^j)^\top (M^j)^\top z^j = c.
\ea \right.
\end{equation}
In what follows, let $(B_{\ref{pb:cop}},N_{\ref{pb:cop}},R_{\ref{pb:cop}},T_{\ref{pb:cop}})$
denotes the optimal partition of \eqref{pb:cop},
and adopt a similar convention for (MP).

Our goal in this section is to extend the results of \cite{BonRam} to the six-partition. This permits to end this section by showing that relevant pair of second-order and semidefinite programming problems has the same partitions. 
%
%
%
Before proceeding in this direction we need to recall some useful claims proved in Lemmas 8 and 9 of \cite{BonRam}.
\begin{lema}
	\label{mcopos.l1} The following relations hold: 
	
	{\rm (i)} $
	S\eqref{pb:cop} = S\eqref{pb:mcopa}$, $M^\top \calk^+ \subset K^+$, and $M^\top
	S\eqref{pb:dmcop}  \subset S\eqref{pb:dcop}$. 
	
	{\rm (ii)} If $M^\top\calk^+$ is
	closed, then $M^\top \calk^+ = K^+$ and $M^\top S\eqref{pb:dmcop} =
	S\eqref{pb:dcop}$. 
	
	{\rm (iii)} Closedness of $M^\top\calk^+$ holds if
	$M^\top$ is coercive on $\calk^+$, i.e., if there exists $\gamma
	>0$ such that $\| M^\top z \| \geq \gamma \|z\|$ for all $z\in
	\calk^+$. In that case, $S\eqref{pb:dmcop} $ is bounded iff $S\eqref{pb:dcop}$
	is bounded.
	
	{\rm (iv)} Assume in addition  that the linear mapping defined by the matrix $M$ is one to one. Then,
	$M^\top \intt \calk^+ = \intt K^+$
	and \\
	$M^\top \intt S\eqref{pb:dmcop} = \intt S\eqref{pb:dcop}$.
	Moreover,   for all $s\in K$,
	$M^\top \ri (\calk^+ \cap (Ms)^\perp) \subset \ri (K^+\cap s^\perp)$.
\end{lema}

The next proposition strengthens \cite[Lemma 10]{BonRam}  by replacing the inclusions $N_{\ref{pb:cop}} \supseteq N_{MP}$ and $T_{\ref{pb:cop}} \subseteq T_{MP}$ with equalities, under the same assumptions. It states conditions for which  the
four-partition of problems \eqref{pb:cop} and \eqref{pb:mcopa} coincide.
\begin{prop}
\label{mcopos.l5}
Assume that $M^\top \calk^+$ is closed, that
$M$ is one to one, and that
\begin{equation}
\label{eq-lema.l5-1}
\text{ For all } s^j \in K_j, \;
M^j s^j \in \bd \calk_j \text{ iff } s^j \in \bd K_j.
\end{equation}
Then, the
four-partition of problems \eqref{pb:cop} and \eqref{pb:mcopa} coincide, that is
\begin{equation}
\label{eq:partition-coincide}
\left\{\begin{array}{lll}
B_{\ref{pb:cop}} = B_{MP}, &
N_{\ref{pb:cop}} = N_{MP},
\\
R_{\ref{pb:cop}} =R_{MP}, &
T_{\ref{pb:cop}} = T_{MP}.
\end{array}\right.
\end{equation}
In particular, the strict complementarity condition holds for
\eqref{pb:cop} iff it holds for \eqref{pb:mcopa}.
\end{prop}

\begin{proof}
As mentioned above, due to Lemma 10 of \cite{BonRam} it only remains to prove that $R_{\ref{pb:cop}} \subseteq R_{MP}$. Let  $j\in R_{\ref{pb:cop}}$. This implies the existence of $(x,y)\in S\eqref{pb:copos}$ such that  $s^j:=A^jx^j-b^j\in \bd K_j\setminus\{0\}$ and $y^j(j)\in \bd K_j^+\setminus\{0\}$.
It follows from Lemma \ref{mcopos.l1}, Parts (i) and (ii), that $x \in S\eqref{pb:mcopa}$ and that there exists $z \in S \eqref{pb:dmcop} $ such that $y=M^\top z $, respectively. In particular, $y^j=(M^j)^\top z^j $ and 
\begin{equation}\label{eq:aux}
0=(s^j) \transp y^j =(s^j) \transp (M^j)^\top z^j =M^j (s^j)\transp z^j.
\end{equation}
Moreover, from \eqref{eq-lema.l5-1} and the fact that $M$ is one to one, we deduce that  $M^js^j$ is a nonzero element of $\bd \K_j$. This in particular implies that $z^j \not\in \Int (\K_j^+)$ (because otherwise  \eqref{eq:aux} yields $M^js^j=0$). Finally, since $(M^j)^\top z^j =y^j \neq 0$ it directly follows that $z^j \neq 0$. Summarizing, we have proved that $(x,z) \in S\eqref{pb:mcopos}$, and $M^js^j$ and $z^j$ are nonzero elements of $\bd \K_j$ and $\bd \K_j^+$, respectively. This means that $j\in R_{MP}$, which concludes our proof.
\end{proof}

We now state  the equivalence between six-partitions of problems  \eqref{pb:cop} and  \eqref{pb:mcopa}.

\begin{teo}\label{lem:17} Under the assumptions of Proposition~\ref{mcopos.l5}, 
	\begin{equation}
		\label{eq:lem17}
		\left\{\begin{array}{lll}
			B_{\ref{pb:cop}} = B_{MP}, &
			N_{\ref{pb:cop}} = N_{MP},
			\\
			N^0_{\ref{pb:cop}} = N^0_{MP}, &
			B^0_{\ref{pb:cop}} \supseteq B^0_{MP}.
		\end{array}\right.
	\end{equation}
	Additionally, suppose that $z=0$ is the only element in  $\K^{+}$ such that $M^\top z =0 $. Then, $B^0_{\ref{pb:cop}} = B^0_{MP}$. The last relation in particular implies that the six-partition of problems  \eqref{pb:cop} and  \eqref{pb:mcopa} coincide.  
\end{teo}



\begin{proof}
Relations $B_{\ref{pb:cop}} = B_{MP}$ and 
			$N_{\ref{pb:cop}} = N_{MP}$ were established in Proposition \ref{mcopos.l5}.
From Lemma~\ref{mcopos.l1}~(i) we have $M$S(\ref{pb:cop}) = S\eqref{pb:mcopa}. Hence, the inclusion $N^0_{\ref{pb:cop}} \subset N^0_{MP}$ is obvious. Moreover, since $M$ is one to one, $M_j(A_j x-b_j)=0$ yields  $A_j x-b_j=0$, and consequently, the equality $N^0_{\ref{pb:cop}} = N^0_{MP}$.

On the other hand, relation $  B^0_{MP}\subseteq B^0_{\ref{pb:cop}}$ follows directly from Lemma~\ref{mcopos.l1}(ii). Indeed, if for every 
$(x,z)\in S$(MC) we have $z_j=0$, then $y_j = M^\top_j(z_j )=0$, obtaining that $y_j =0$ for all $y\in S$(D).
Now, for the opposite inclusion, assume that $j\in B^0_P$. Then every solution $(x,z)\in S\eqref{pb:mcopa}$ satisfies 
$M_j^\top  z_j =0$. By assumption  this yields $z_j =0$. We thus conclude that  $ B^0_{\ref{pb:cop}}  \subseteq B^0_{MP}$. The result follows.
\end{proof}

\if{

The proof of Lemma~\ref{lem:17} relies on two technical claims that we prove next.

\begin{prop}\label{lem:16} Assume that $M^\top \calk^+$ is closed. Then 
$$
B^0_{\ref{pb:cop}} \supset B^0_{MP}.
$$
If, in addition, for all $z\in \K^{+}$ such that $M^\top z =0 $ yields $z=0$, then 
$$
B^0_{\ref{pb:cop}} = B^0_{MP}.
$$
\end{prop}
\begin{proof}
The relation $B^0_{\ref{pb:cop}} \supset B^0_{MP}$ follows directly from Lemma~\ref{mcopos.l1}(ii). If for every 
$(x,z)\in S(MC)$ we have $z_j=0$, then $y_j = M^\top_j(z_j )=0$.

Now assume that $j\in B^0_{\eqref{pb:cop}}$. Then every solution $(x,z)\in S(MC)$ satisfies 
$M_j^\top  z_j =0$. By assumption  this yields $z_j =0$.
\end{proof}

\begin{prop}\label{lem:18} We have 
$$
N^0_{\ref{pb:cop}} \subset N^0_{MP}.
$$
If, in addition, $M$ is one to one, then
$$
N^0_{\ref{pb:cop}} = N^0_{MP}.
$$
\end{prop}
\begin{proof} From Lemma~\ref{mcopos.l1}~(i) we have $MS(\ref{pb:cop}) = S(MP)$, hence, the first relation $N^0_{\ref{pb:cop}} \subset N^0_{MP}$ is obvious. Since $M$ is one to one, $M_j(A_j x-b_j)=0$ yields  $A_j x-b_j=0$, hence, the equality $N^0_{\ref{pb:cop}} = N^0_{MP}$.
\end{proof}

\begin{proof}[Proof of Lemma~\ref{lem:17}]
Follows directly from Propositions~\ref{lem:16} and \ref{lem:18}.
\end{proof}
}\fi

We show in the next example that the assumptions of  Theorem~\ref{lem:17} are crucial for the equality between partitions.
\begin{ej}
	
	Consider a pair of feasibility problems
	$$
	Ax \in K, \quad (P) \qquad A^\top y = 0, y\in K^+,\quad (D)
	$$
	where
	$$
	K = K^+ = L_2\times \rr_+^1,
	$$
	and $L_2$ is the three-dimensional Lorentz cone.
	Let 
	$$
	A^\top  = \left[
	\begin{array}{cccc}
	0 & 1 & 0 & 1\\
	1 & 0 & 1 & 0
	\end{array}
	\right].
	$$
	We can solve this feasibility problem directly and obtain the parametric family of solutions
	$$
	x^* = (0,\alpha)^\top , \alpha \geq 0, \text{ with } Ax^* = (\alpha, 0,\alpha, 0)^\top ; \qquad 
	y^* = (\beta, 0, -\beta, 0)^\top , \beta \geq 0.  
	$$
	It is not difficult to see that for the first index the solutions lie on the boundary of the cones, hence, $1\in C=R$.
	For the second index, both primal and dual components are zero, hence, $2\in O$.

	Now we transform our problem into a higher-dimensional one. We let 
	$$
	M_1 = I_3  = \left[\begin{array}{ccc}1 & 0 & 0 \\ 0 & 1 & 0 \\ 0 &0 & 1\end{array}\right], \quad M_2 = \left[\begin{array}{c} 1 \\ 1\end{array}\right].
	$$
	For the transformed problem, we let $\K_1 = K_1 = L_2$, $\K_2 = L_1$.
	Observe that for every $y\in \rr$ we have 
	$$
	M_2 y = (y,y)^\top ,
	$$
	and hence $M_2 y \in \K_2= L_1$ iff $y\in K_2=\rr_+^1$.
	Further,
	$$
	\A_1 = M_1 A_1 = A_1 = \left[
	\begin{array}{ccc}
	0& 1 \\
	1 & 0 \\
	0 & 1
	\end{array}
	\right], \quad \A_2  = M_2 A_2 = \left[
	\begin{array}{ccc}
	1 \\
	1 
	\end{array}
	\right] [1\; 0] = \left[
	\begin{array}{cc}
	1& 0 \\
	1 & 0 
	\end{array}
	\right].
	$$
	Let $\A = [\A_1;\A_2]$, then we can write the transformed feasibility problem as
	$$
	\A u \in \K, \quad (P') \qquad \A^\top v = 0, v\in \K^+.\quad (D')
	$$
	Solving the problem directly, we obtain the family of optimal solutions
	$$
	u = (\lambda, 0), \lambda \geq 0 \text{ with } \A u  = (\lambda, 0 , \lambda, 0, 0); \qquad v= (\mu,0,-\mu, \gamma, -\gamma), \mu, \gamma \geq 0.
	$$
	Therefore, we conclude that $1_M\in C$, $2_M \in N^0 \setminus (B^0 \cup N) = N'$, and hence the partition has changed. 
	
	The situation above occurs because $M_2$ does not satisfy that $M_2^\top z =0 $ yields $z=0$ for all $z\in \K^{+}=L_1$. Indeed, it is enough to take $z=(1,-1)^\top$ to check this hypothesis fails.
\end{ej}

The next corollary  follows directly from Theorem~\ref{lem:17} and definition of the index set $C$. This result provides alternative hypotheses to those in Proposition \ref{mcopos.l5} in order to ensure that the four-partition of problems (P) and (MP) coincide.

\begin{cor}\label{lem:19} Suppose that the assumptions of Lemma~\ref{mcopos.l5} are fulfilled. Suppose also that for all $z\in \K^{+}$ such that $M^\top z =0 $ yields $z=0$,  $R_{\ref{pb:cop}}=C_{\ref{pb:cop}}$ and $R_{MP}=C_{MP}$. Then 
	$$
	R_{\ref{pb:cop}} = R_{MP}, \quad T_{\ref{pb:cop}} = T_{MP}.
	$$ 
\end{cor}
%

\if{
\begin{lema}\label{lem:equivSOCSDP}
	Problems $(LSOCP)$ and $(LSDP)$ have the same four and six-partition. 
\end{lema}

}\fi

\medspace


We end this section by applying previous results to the case of a linear  second-order cone program and its semidefinite programming equivalent representation constructed via \eqref{eq:equivArw}.
Indeed, those  problems fits  in the framework given by problems (P) and (MP) provided, for all $j\in J$, we consider the following sets
$K_j:=L_{n_j}$, $\calk_j := \SS^{n_j+1}_+$,
and $M^j s^j = \Arw^j s^j$. Here, the arrow function $\Arw^j: \rr^{n_j+1}\to \SS^{n_j+1}$
was defined in \eqref{eq:equivArw}.

In what follows consider a generic nonnegative integer value $n$ and omit index $j$ from the arrow function.
Note that we can write
\begin{equation}
\label{propm}
\Arw(s) = (s_0 - \| \sb \|) I_{n+1} +
\begin{pmatrix} \| \sb \| & \sb^\top
                   \\ \sb & \| \sb \| I_n
\end{pmatrix}.
\end{equation}
This shows that for $s\in L_{n}\setminus \{0\}$,
$\Arw(s)$ is of rank $n$ iff
$s \in \bd L_{n}$, and of rank $ n+1$ otherwise.
In particular,
$\Arw \bd L_{n} \subset  \bd \SS^{n+1}_+$, and
$\Arw \intt L_{n} \subset  \intt \SS^{n+1}_+$.
Therefore
\eqref{eq-lema.l5-1}
holds. Also $\Arw$ is clearly one to one.

Let us decompose any matrix $Y\in \SS^{n+1}$ as follows
\begin{equation}
\label{eq:decomposition_Y}
Y=\left(\begin{matrix} Y_{00} & \y0^\top\\ \y0 & \bar Y \end{matrix}\right) ,
\end{equation}
where $ Y_{00}\in \rr$, $ \y0 \in \rr^n$ and $ \bar Y \in \SS^n$. We
note that for any $s\in \rr^{n+1}$ we get 
\begin{equation}\label{eq:operation_Arw}
\Arw(s)\cdot Y= s_{0}\Tr(Y) +2 \sb\cdot\y0.
\end{equation}
It follows that
$\Arw^\top:\SS^{n+1}\rar \rr^{n+1}$
is nothing but
\begin{equation}
\label{eq:defarwt}
\mbox{$\Arw^\top$}
Y:= \left( \begin{matrix}  \Tr( Y) \\ 2 \y0  \end{matrix}\right).
\end{equation}
Consequently
\begin{equation}
\label{eq:defarwt2}
M^\top (Y^1,\ldots,Y^r) = \vec
\left(
\left( \begin{matrix}  \Tr( Y^1) \\ 2 \y0^1  \end{matrix}\right),
\ldots,
\left( \begin{matrix}  \Tr( Y^r) \\ 2 \y0^r  \end{matrix}\right)
\right).
\end{equation}
Then, $M^{\top}\calk$ is clearly closed. Moreover, since $\Arw^\top Y=0$ implies $\Tr(Y) =0$, when $Y\in \SS^{n+1}_{+}$ the latter yields $Y=0$. Hence, all the assumptions of Proposition \ref{mcopos.l5},  and of Theorem \ref{lem:17}, are fulfilled. We thus conclude that
a second-order program and its respective semidefinite programming representation (given by \eqref{eq:equivArw}) have the same four- and six-partitions. 


\section{Conclusions}

Our results demonstrate what one cannot expect to obtain simple geometric characterizations of the kind that exist for the four-partition of Bonnans and Ram\'irez  \cite{BonRam} for the practically useful partition of Pe\~na and Roshchina  \cite{PR2012} in a general case of a multifold conic system. However, in the case of second-order cone programs the complement $T\setminus C$ is empty (cf. Fig.~\ref{fig:venn}), and $R=C$ (see Proposition~\ref{lem:partsoc}). This is in contrast to Example~4 in which we lumped three copies of the nonnegative orthant together, deliberately disrespecting the natural multifold structure of the problem.

In the case of a homogeneous feasibility problem, we have obtained a geometric condition satisfied by the set $R^0$ that originates from the six-partition of Bonnans and Ram\'irez. It would be valuable to generalize this result (stated in Theorem~\ref{thm:RDualChar}) as well as the original geometric result (in Theorem~\ref{thm:BNDualChar}) to optimization problems. 

Also, motivated by a question from one of the referees that we were not able to answer, we would like to ask the reader if there is an example of a primal-dual pair of problems for which $B'\cup N'\neq \emptyset$, but $C=\emptyset$.

Finally, we have shown that the canonical transformation of second-order cone programs into a positive semidefinite program preserves the complementarity partitions. It is an interesting question to determine other specific classes of problems for which this is also the case. 

\if{

\begin{prop}
\label{applic}
{\rm (i)}
We have that $y$ is solution of $(LSOCP^*)$ iff there exists
$z$ solution of $(LSDP^*)$ such that $y=M^\top z$.
{\rm (ii)}
One of these dual problems has a bounded set of solutions iff the
other one has the same property.
{\rm (iii)}
One of these dual problems has an interior feasible point
iff the other one has the same property.
{\rm (iv)}
Problems $(LSOCP)$ and $(LSDP)$ have the same optimal partition.
\end{prop}

\begin{proof}
Since
$\Arw^\top$ is coercive on $\SS^{m+1}_+$,
$M^\top$ is also coercive.
By lemma \ref{mcopos.l1}, we have that
$S(LSOCP^*) = M^\top S(LSDP^*)$ and
$S(LSDP^*)$ is bounded iff $S(LSOCP^*)$ is bounded.
This proves points i) and (ii).
Point (iii) is consequence of lemma  \ref{mcopos.l2}(ii).
We now prove (iv).
By lemma \ref{mcopos.l5},
$B_{LSOCP} = B_{LSDP}$, $N_{LSOCP} = N_{LSDP}$, and
$R_{LSOCP} \supset R_{LSDP}$;
it remains to prove that $R_{LSOCP} \subset R_{LSDP}$
since $(B,N,R,T)$ is a partition.
Let $j\in R_{LSOCP}$. Then there is a pair
$(x,y)$ solution of \eqref{KKT:linear_socp}
such that $s^j\neq 0\neq y^j$, and both $s^j$ and $y^j$
belong to the boundary of $Q_{m^j+1}$.
As observed after \eqref{propm},
this implies that
$\Arw s^j$ is of rank $m^j$, and hence, the corresponding set
of normals is a half line (of rank one positive semidefinite matrices,
orthogonal to $\Arw s^j$). Since the corresponding multiplier
$Y$ for problem $(LSDP)$ is such that $0 \neq y^j = \Arw^\top Y^j$,
we have that $Y^j \neq 0$, proving that
 $-Y^j$ belongs to the relative interior of the normal
cone (to the set of positive semidefinite matrices) at $\Arw s^j$.
\end{proof}

The above analysis shows that strong duality holds for problem
\eqref{pb:socp} iff it holds for problem \eqref{pb:lsdp}.
The next proposition states an interesting relation between the
solutions of \eqref{pb:socp_dual} and \eqref{pb:sdp_dual}.


\begin{prop}
\label{prop:SZhao3}
Let the strong duality property hold for problem \eqref{pb:socp}.
Let
$I$ be the set of indices in  $J$ such that there exists
$x^* \in S\eqref{pb:socp}$ satisfying
$A^jx^*\neq b^j$.
Then every $Y \in S\eqref{pb:sdp_dual}$
is such that, for some $y \in S\eqref{pb:socp_dual}$,
the following relation holds:
\begin{equation}\label{eq:sol_Y^*}
Y^j=  \, 0, \, \textrm{ if } y^j = 0; \qquad
Y^j=\half \left( \begin{matrix}
\|\yb^j\| &(\yb^j)^\top \\
\yb^j & \yb^j (\yb^j)^\top /  \|\yb^j\|
\end{matrix}\right), \,\textrm{ otherwise.}
 \end{equation}
\end{prop}

\begin{proof}
Let $j\in I$, $x^*$ be the associated solution of
$\eqref{pb:socp}$, and let $Y \in S\eqref{pb:sdp_dual}$.
We claim  that
\begin{equation}\label{eq:null_dec_Y^*}
Y^j_{00}\bar Y^j -(\y0^j)({\y0}^j)^\top= 0,
\end{equation}
where $Y^j_{00}$, $\bar Y^j$ and $\bar Y^j_0$ are given by
\eqref{eq:decomposition_Y}. Since
$Y^j\in \SS^{m_j+1}_+$, by Schur complement the matrix
$Y^j_{00}\bar Y^j -(\bar Y^j_0)(\bar Y^j_0)^\top$ is positive semidefinite,
and hence, it is enough to show that
\begin{equation}\label{eq:trace_dec_Y^*}
\Tr\left( Y^j_{00}\bar Y^j -(\bar Y^j_0 )(\bar Y^j_0)^\top \right) \leq 0.
\end{equation}
By strong duality, any primal-dual solution $(x^*,y^*)$
of \eqref{pb:socp} is solution of
$\eqref{KKT:linear_socp}$.
Since $A^jx^*\neq b^j$,  the complementarity condition
implies that any
$y \in S\eqref{pb:socp_dual}$
satisfies $y^j_0=\|\yb^j\|$.
Taking $y^j = \Arw^\top Y^j$,
we deduce
$
\Tr(Y^j)= y^j_0=\|\yb^j\|= 2\|\bar Y^j_0\|, $
 which implies
\begin{equation}
\label{eq:aux_trace_Y^*}
\begin{array}{ll}
\Tr\left( Y^j_{00}\bar Y^j -(\bar Y^j_0)(\bar Y^j_0)^\top\right)
& =
Y^j_{00}\Tr\left( Y^j\right) -(Y^j_{00})^2 - \|\bar Y^j_0\|^2
\\
& =
-\left(Y^j_{00}- \|\bar Y^j_0\|\right)^2\leq 0,
\end{array}
\end{equation}
proving \eqref{eq:trace_dec_Y^*} and therefore also
\eqref{eq:null_dec_Y^*}.
Combining \eqref{eq:defarwt} and \eqref{eq:aux_trace_Y^*},
obtain
\begin{equation}\label{eq:aux1_trace_Y^*}
 Y^j_{00}=\|\bar Y^j_0\|=\half \|\bar y^j\|.
 \end{equation}\noindent Now, we distinguish two cases:
a) If $Y^j_{00}=0$, we obtain from \eqref{eq:aux1_trace_Y^*} that
 $\bar Y^j_0=\yb^j=0$ and then $\Tr(Y^j)=y_0^j=0$.
Hence, since $Y^j$ is positive semidefinite
this implies $Y^j=0$.
b) Else if $Y^j_{00}\neq 0$, we get directly from
\eqref{eq:null_dec_Y^*} and \eqref{eq:aux1_trace_Y^*} that
$$
\bar Y^j=(Y^j_{00})^{-1}(\bar Y^j_0)(\bar Y^j_0)^\top
=\frac{2}{\|\yb^j\|}(\bar
y^j/2)(\yb^j/2)^\top= \half (\yb^j)(\yb^j)^\top/\|\yb^j\|,
$$
which, combined with \eqref{eq:aux1_trace_Y^*},
allows to conclude the proof.
\end{proof}

}\fi


\medskip

{\bf Acknowledgements.}
We thank to two anonymous reviewers who have substantially contributed to the improved quality of the revision.
This research was partially supported by ANID (Chile) under REDES project number 180032 and by Australian Research Council grant DE150100240.
The second author was supported by FONDECYT regular projects 1160204 and 1201982, and Basal Program CMM-AFB 170001, all from ANID (Chile).

\end{document}